\documentclass[reqno]{amsart}

\usepackage{hyperref}

\newcommand*{\mailto}[1]{\href{mailto:#1}{\nolinkurl{#1}}}

\newtheorem{theorem}{Theorem}[section]
\newtheorem{lemma}[theorem]{Lemma}

\newcommand{\R}{{\mathbb R}}

\newcommand{\C}{{\mathbb C}}
\newcommand{\cA}{{\mathcal A}}
\newcommand{\cR}{{\mathcal R}}
\newcommand{\om}{\omega}

\newcommand{\nn}{\nonumber}

\newcommand{\I}{\mathrm{i}}
\newcommand{\E}{\mathrm{e}}

\numberwithin{equation}{section}

\begin{document}

\title[Zero Energy Scattering for Schr\"odinger Operators and Applications]{Zero Energy Scattering for One-Dimensional Schr\"odinger Operators and Applications to Dispersive Estimates}

\author[I.\ Egorova]{Iryna Egorova}
\address{B. Verkin Institute for Low Temperature Physics\\ 47, Lenin ave\\ 61103 Kharkiv\\ Ukraine}
\email{\href{mailto:iraegorova@gmail.com}{iraegorova@gmail.com}}

\author[M.\ Holzleitner]{Markus Holzleitner}
\address{Faculty of Mathematics\\ University of Vienna\\Oskar-Morgenstern-Platz 1\\ 1090 Wien\\ Austria}
\email{\href{mailto:amhang1@gmx.at}{amhang1@gmx.at}}

\author[G.\ Teschl]{Gerald Teschl}
\address{Faculty of Mathematics\\ University of Vienna\\
Oskar-Morgenstern-Platz 1\\ 1090 Wien\\ Austria\\ and International Erwin Schr\"odinger
Institute for Mathematical Physics\\ Boltzmanngasse 9\\ 1090 Wien\\ Austria}
\email{\mailto{Gerald.Teschl@univie.ac.at}}
\urladdr{\url{http://www.mat.univie.ac.at/~gerald/}}

\thanks{Proc. Amer. Math. Soc. Ser. B {\bf 2}, 51--59 (2015)}
\thanks{{\it Research supported by the Austrian Science Fund (FWF) under Grants No.\ Y330 and W1245}}


\begin{abstract}
We show that for a one-dimensional Schr\"odinger operator with a potential, whose $(j+1)$'th moment is integrable, the $j$'th derivative of the scattering matrix is in the Wiener algebra of functions with integrable Fourier transforms. We use this result to improve the known dispersive estimates with integrable time decay for the one-dimensional Schr\"odinger equation
in the resonant case.
\end{abstract}

\maketitle

\section{Introduction}

This paper is concerned with the one-dimensional Schr\"odinger equation
\begin{equation} \label{Schr}
  \I \dot \psi(x,t)=H \psi(x,t), \quad H=-\frac{d^2}{dx^2}  + V(x),\quad (x,t)\in\R^2,
\end{equation}
with a real-valued potential $V$ contained in one of the spaces
$L^1_{\sigma}=L^1_{\sigma}(\R)$, $\sigma\in\R$, associated with the norms
\[
   \Vert \psi\Vert_{L^p_{\sigma}}= \begin{cases} \left(\int_{\R}(1+|x|)^{p\sigma} |\psi(x)|^p dx\right)^{1/p}, & 1\le p <\infty,\\
   \sup_{x\in\R} (1+|x|)^{\sigma} |\psi(x)|, & p=\infty.
   \end{cases}
\]
We recall (e.g., \cite{Mar} or \cite[Sect.~9.7]{tschroe}) that for $V\in L^1_1$ the operator $H$ has a purely absolutely continuous spectrum
on $[0,\infty)$ and a finite number of eigenvalues in $(-\infty,0)$. Associated with the absolutely continuous spectrum is the scattering matrix
\[
\mathcal{S}(k) = \begin{pmatrix} T(k) & R_-(k)\\ R_+(k) & T(k) \end{pmatrix},
\]
which maps incoming to outgoing states at a given energy $\om=k^2\ge 0$. Here $T$ is the transmission coefficient and
$R_\pm$ are the reflection coefficients with respect to right and left incident. At the edge of the continuous spectrum $k=0$ the scattering matrix generically
looks like
\[
\mathcal{S}(0) = \begin{pmatrix} 0 & -1\\ -1 & 0 \end{pmatrix}.
\]
More precisely, this happens when the zero energy is non-resonant, that is, if  the equation $H f_0 =0$ has no bounded (distributional) solution. In the resonant situation
the behavior of the scattering matrix is more delicate. For $V\in L^1_1$ it is already nontrivial to establish continuity of the scattering matrix at $k=0$ (for
$V\in L^1_2$ de l'Hospital's rule suffices).
This question arose around 1985 in an attempt to clarify whether the low-energy asymptotics of the scattering matrix,
 obtained for $V\in L_2^1$ in \cite{DT}, are valid for $V\in L_1^1$. The problem was solved independently
by Guseinov \cite{Gus2} and Klaus \cite{Kl} (for a refined version see \cite{AK} and \cite[Theorem~2.1]{EKMT}).
It also plays an important role in semiclassical analysis, see \cite{CSST}
and the references therein. Furthermore, if $V\in L^1_{j+1}$ with $j>0$, then, away from $0$, one can take derivatives of the scattering matrix up to order $j$, and in the non-resonant case
it is again easy to see that they are continuous at $k=0$. This clearly raises the question about continuity at $k=0$ of the $j$'th derivative in the resonant case. We will establish this as one of
our main results in Theorem~\ref{MT}. This result is new even for the first derivative. We remark, that if $V$ decays exponentially, then $\mathcal{S}(k)$ is
analytic in a neighborhood of $k=0$, and the full Taylor expansion can be obtained \cite{BGK,BGW}.

It is important to emphasize that this question is not only of interest in scattering theory, but also plays a role in solving the Korteweg--de Vries equation
via the inverse scattering transform (see e.g., \cite{GT}, where continuity of higher derivatives is needed), and in deriving dispersive estimates for \eqref{Schr}.
The latter case has attracted considerable interest (e.g.\ \cite{EKMT,gold} and the references therein) due to its importance for proving asymptotic stability of
solitons for the associated nonlinear evolution equations (see e.g.\ \cite{BS,KK11}).

As an application of our results, we will establish the following dispersive decay estimate with integrable time decay in the resonant case:

\begin{theorem} \label{thm:mainr}
Suppose $V \in L_3^1(\R)$ and suppose there is a bounded solution $f_0$ of $H f_0=0$ satisfying the normalization $\lim_{x \to \infty} (|f_0(x)|^2 + |f_0(-x)|^2)=2$.
Denote by $P_0:L^1_2 \to L^\infty_{-2}$ the operator given by the kernel $[P_0](x,y)=f_0(x)f_0(y)$. By $P_{ac}$ we denote the projector onto the absolutely
continuous subspace of $H$. Then $\E^{-\I tH}P_{ac}$ extends to a bounded operator $L^1_2 \to L^\infty_{-2}$ satisfying the following decay estimate:
\begin{equation} \label{resdec}
\Vert \E^{-\I tH}P_{ac}-(4 \pi \I t)^{-\frac{1}{2}}P_0 \Vert_{L^1_2 \to L^\infty_{-2}}=\mathcal{O}(t^{-3/2}),\quad t\to\infty .
\end{equation}
\end{theorem}

This theorem is an improvement of an earlier result by Goldberg \cite{gold}, who established it for $V \in L_4^1(\R)$.
If there is no resonance (i.e.\ no bounded solution) this result (with $P_0=0$) was shown for $V \in L_2^1(\R)$ in \cite[Theorem~1.2]{EKMT}.
For extensions to discrete one-dimensional Schr\"odinger equations (Jacobi operators) see \cite{EHT}, \cite{EKT}.

\section{Low energy scattering}
\label{sec1}

In this section we establish some properties of the scattering data for our operator $H$ with $V\in L_{j+1}^1$, $j\geq 0$.
To this end we introduce the Banach algebra $\cA$ of Fourier transforms of integrable functions:
\begin{equation}\label{cA}
\cA = \left\{f(k):\,
f(k) = \int_\R \E^{\I k p}\hat{f}(p)dp, \,\hat{f}(\cdot)\in L^1(\R) \right\}
\end{equation}
with the norm $\|f\|_{\cA}= \|\hat{f}\|_{L^1}$, plus the corresponding unital Banach algebra $\cA_1$:
\begin{equation}\label{cA1}
\cA_1 = \left\{f(k):\,
f(k) =c+ \int_\R \E^{\I k p}\hat{g}(p)dp, \,\hat{g}(\cdot)\in L^1(\R),\,c\in\C\right\}
\end{equation}
with the norm $\|f\|_{\cA_1}= |c|+\|\hat{g}\|_{L^1}$. We also use the fact, which is known as Wiener's lemma \cite{Wiener}, that if $f\in\cA_1\setminus\cA$ and $f(k)\not =0$ for all $k\in\R$ then $f^{-1}(k)\in\cA_1$.

First we recall a few facts from scattering theory \cite{DT}, \cite{Mar}.
If $V\in L^1_1$, there exist Jost solutions $f_\pm(x,k)$ of $H f=k^2 f$, $k\in \overline{\C_+}$, which asymptotically behave like $f_\pm(x,k)\sim  \E^{\pm \I kx}$ as $x\to \pm \infty$. These solutions are given by
\begin{equation}\label{Jost}
f_\pm(x,k) = \E^{\pm \I kx} h_\pm(x,k), \qquad h_\pm(x,k)
= 1 \pm \int_{0}^{\pm\infty} B_\pm(x,y) \E^{\pm 2\I k y}dy,
\end{equation}
where $B_\pm(x,y)$ are real-valued and satisfy (see \cite[\S 2]{DT} or \cite[\S 3.1]{Mar})
\begin{align}\label{est1}
|B_\pm (x,y)|&\le  \E^{\gamma_\pm(x)}\eta_\pm(x+y),\\
\label{est11}
|\frac{\partial}{\partial x}B_\pm (x,y)\pm V(x+y)|
&\le 2 \E^{\gamma_\pm(x)}\eta_\pm(x+y)\eta_\pm(x),
\end{align}
with
\begin{equation}\label{est2}
\gamma_\pm(x)=\int_{x}^{\pm\infty}(y-x)|V(y)|dy,\quad
\eta_\pm(x)=\pm\int_{x}^{\pm\infty}|V(y)|dy.
\end{equation}
Since $\eta_\pm\in L^1(\R_\pm)$, we conclude that
\begin{equation}\label{ast}
h_\pm(x,\cdot) - 1, \quad h'_\pm(x,\cdot)\in\mathcal A,
\quad \forall x\in\R.
\end{equation}
Here and throughout the rest of this paper a prime will always denote a derivative with respect to the spatial variable $x$.
As an immediate consequence of the estimates \eqref{est1} and \eqref{est11} we have the following strengthening of \eqref{ast}:

\begin{lemma} \label{lem1rc}
Let $V \in L_{j+1}^1$. Then $\frac{\partial^l}{\partial k^l} (h_\pm(x,k)-1)$, $\frac{\partial^l}{\partial k^l} h'_\pm(x,k)$ are contained in $\mathcal A$ for $0\le l \le j$.
Moreover, for $\pm x \geq 0$, the $\mathcal A$-norms of these expressions do not depend on $x$.
\end{lemma}

The fact that $f_\pm(x,-k)$ also solve $H f=k^2 f$ for $k\in\R$ leads to the scattering relations
\begin{equation}\label{scat-rel}
T(k)f_\pm(x,k)=R_\mp(k)f_\mp(x,k) + f_\mp(x,-k),
\end{equation}
where the transmission coefficient $T$ and the reflection coefficients $R_\pm$ can be expressed in terms of Wronskians.
To this end we denote by
\[
W(f(x),g(x))= f(x)g'(x)-f'(x)g(x)
\]
the usual Wronskian and set
\[
W(k)=W(f_-(x,k),f_+(x,k)),\qquad W_\pm(k)=W(f_\mp(x,k),f_\pm(x,- k)).
\]
Then
\begin{equation}\label{defRR}
T(k)= \frac{2\I k}{W(k)},\quad R_\pm(k)= \mp\frac{W_\pm(k)}{W(k)}.
\end{equation}
The transmission and reflection coefficients
are elements of the Wiener algebra, which was established in \cite[Theorem 2.1]{EKMT}. Here we extend this result to the derivatives of the scattering data.

\begin{theorem}\label{MT}
If $V\in L^1_{j+1}$, then $\frac{d^l}{dk^l}(T(k)-1)\in\cA$ and $\frac{d^l}{dk^l}R_\pm(k)\in\cA$ for $0\le l \le j$.
\end {theorem}

\begin{proof}
We  only focus on the resonant case, since the other case is straightforward.

First of all, we abbreviate $h_\pm(k)=h_\pm(0,k)$,  $h'_\pm(k)=h'_\pm(0,k)$. Then \eqref{Jost} leads us to
\begin{align}\label{Wr}
W(k) &=2\I k h_+(k)h_-(k)+ h_-(k)h'_+(k)-h'_-(k)h_+(k),\\\label{Wrpm}
W_\pm(k) &=h_\mp(k)h'_\pm(-k)-h_\pm(-k)h'_\mp(k).
\end{align}
Following \cite{EKMT} we introduce
\begin{equation}\label{Phi}
\Phi_\pm(k)=h_\pm(k)h'_\pm(0)-h'_\pm(k)h_\pm(0),
\end{equation}
\begin{equation}\label{defK}
K_\pm(x,y)=\pm\int_y^{\pm\infty}B_\pm(x,z)dz,\quad
D_\pm(x,y)=\pm\int_y^{\pm\infty}\frac{\partial}{\partial x}B_\pm(x,z)dz,
\end{equation}
where $B_\pm(x,y)$ are given by \eqref{Jost}. Again $K_\pm(0,y)$ is denoted by $K_\pm(y)$ and similar for $D_\pm(y)$. In \cite[Theorem 2.1]{EKMT}, the following equation for $\Phi_\pm(k)$ was obtained:
\begin{equation} \label{defphipsi}
\Phi_\pm(k)=2\I k\Psi_\pm(k),\quad
\Psi_\pm(k)=\int_0^{\pm\infty}H_\pm(y)\E^{\pm 2\I ky}dy,
\end{equation}
where
\begin{equation} \label{fctH}
H_\pm(x)=D_\pm(x)h_\pm(0)-K_\pm(x)h'_\pm(0),\quad \pm x\geq 0.
\end{equation}
Moreover, $H_\pm$ satisfies an estimate similar to \eqref{est1}, as we will show in Lemma~\ref{lem3rc} below. As a consequence,
$\Psi_\pm$ and its derivatives up to order $j$ will be in the Wiener algebra.

Next, a straightforward computation (cf.\ \cite{EKMT}) shows
\[
\frac{W(k)}{2\I k}= h_-(k) h_+(k) + \begin{cases} \frac{h_+(k)}{h_-(0)}\Psi_-(k)-\frac{h_-(k)}{h_+(0)}\Psi_+(k), & h_+(0)h_-(0)\not =0,\\
\frac{h'_+(k)}{h'_-(0)}\Psi_-(k)-\frac{h'_-(k)}{h'_+(0)}\Psi_+(k), & h_+(0)=h_-(0) =0,\end{cases}
\]
and since $\frac{W(k)}{2\I k} =T(k)^{-1} \ne 0$ for all $k\in\R$, and $T(k)\to 1$ as $k\to\infty$, we conclude that $\frac{d^l}{dk^l}(T(k)-1)\in\cA$ for $0\le l \le j$.
Analogously,
\[
\frac{W_\pm(k)}{2\I k}= \begin{cases}
\frac{h_\pm(-k)}{h_{\mp}(0)}\Psi_{\mp}(k)-\frac{h_\mp(k)}{h_\pm(0)}\Psi_\pm(-k) , & h_+(0)h_-(0)\not =0,\\
\frac{h'_\pm(-k)}{h'_\mp(0)}\Psi_\mp(k)-\frac{h'_\mp(k)}{h'_\pm(0)}\Psi_\pm(-k), & h_+(0)=h_-(0) =0,\end{cases}
\]
and hence $R_\pm(k) = \mp \frac{W_\pm(k)}{2\I k} T(k)$ has the claimed properties.
\end{proof}

To complete the proof of Theorem~\ref{MT} we need the following result, which is an extension of \cite[Lemma ~2.2]{EKMT}:

\begin{lemma}\label{lem3rc}
Let $H_\pm(x)$ be given by \eqref{fctH}. For
 $V\in L^1_1$ the folowing estimate is valid:
\begin{equation}\label{Hhh}|H_\pm(x)| \leq \hat{C} \eta_\pm (x),\quad \pm x\geq 0,\end{equation} with some  constant $\hat{C}>0$ and $\eta_\pm(x)$ given by \eqref{est2}.
 Moreover, for $\Psi_\pm(k)$ defined by  \eqref{defphipsi} and for $V\in L^1_{j+1}$
   we have $$\frac{d^l}{dk^l} \Psi_\pm(k)\in \cA,\quad 0\le l \le j.$$
\end{lemma}

\begin{proof}
It suffices to prove the estimate for $H_\pm$.
The Marchenko equation (\S 3.5 in \cite{Mar}) states that the kernels $B_\pm(x,y)$ solve the equations
\begin{equation}\label{GLM}
F_\pm(x+y)+B_\pm(x,y)\pm\int_0^{\pm\infty}B_\pm(x,t)F_\pm(x+y+z)dz=0,
\end{equation}
where the functions $F_\pm(x)$ are absolutely continuous with $x F'_\pm(x)\in L^1(\R_\pm)$ and
\begin{equation}\label{FF}
|F_\pm (x)|\le C\eta_\pm(x),\quad \pm x\ge 0,
\end{equation}
with $\eta_\pm$ from \eqref{est2}. Now the calculations in \cite[Lemma ~2.2]{EKMT} lead to the following integral equation for $H_\pm(x)$:
\begin{align}\label{13}
H_\pm(x)\mp & \int_0^{\pm\infty}H_\pm(y)F_\pm(x+y)dy=G_\pm(x), \\
G_\pm(x)=h_\pm(0)\Big( & \int_0^{\pm\infty}B_\pm(0,y)F_\pm(x+y)dy \pm  F_\pm(x)\Big). \nn
\end{align}
The estimates \eqref{est1} and \eqref{FF} imply
\begin{equation}\label{G-est}
|G_\pm(x)|\le \widetilde{C} \eta_\pm(x),\quad \pm x\ge 0.
\end{equation}
Now, for a given  potential $V$  fix $N>0$ such that
\[
D(N):=\max_{\pm}\left(\pm C \int_{\pm N}^{\pm\infty} \eta_\pm (y) dy \right)\in (0,1),
\]
where $C$ is given by $\eqref{FF}$. Then we can rewrite \eqref{13} in the form
\begin{align}\label{14}
H_\pm(x)\mp &\int_{\pm N}^{\pm\infty}H_\pm(y)F_\pm(x+y)dy=G_\pm(x,N),\\
G_\pm(x,N)&=G_\pm(x)\pm \int_0^{\pm N}H_\pm(y)F_\pm(x+y)dy. \nn
\end{align}
From \eqref{defK} and the estimates \eqref{est1}--\eqref{est11} we deduce
$H_\pm\in L^{\infty}(\R_\pm)\cap C(\R)$.
We also have
\begin{equation} \label{GNest}
|G_\pm(x,N)|\le C(N)\eta_\pm(x)
\end{equation}
by the boundedness of $H_\pm$, the estimate \eqref{G-est}, and monotonicity of $\eta_\pm(x)$.
Applying to \eqref{14} the method of successive approximations we get
\[
|H_\pm(x)| \leq \ C(N) \eta_\pm (x) \sum_{n=0}^{\infty} \left( \pm C \int_{\pm N}^{\pm\infty} \eta_\pm (y) dy \right)^n \leq \hat{C}(N) \eta_\pm (x),
\]
with $\hat C(N)=C(N)(1 - D(N))^{-1}$. This implies \eqref{Hhh}. The rest follows from \eqref{defphipsi}. 
\end{proof}

For later use we note that in the resonant case the Jost solutions are dependent at $k=0$. If we define $\gamma$ via
\begin{equation}\label{def:gam}
f_+(x,0) = \gamma f_-(x,0),
\end{equation}
 then a straightforward calculation using the scattering relations \eqref{scat-rel} as well as $|T(k)|^2+|R_\pm(k)|^2=1$
shows
\begin{equation}\label{eq:trz}
T(0)= \frac{2\gamma}{1+\gamma^2}, \qquad R_\pm(0)= \pm\frac{1-\gamma^2}{1+\gamma^2}.
\end{equation}
In particular, all three quantities are real-valued since $f_\mp(x,0)\in\R$ and hence $\gamma\in\R$.

To establish Theorem \ref{thm:mainr} we need the following generalization of Lemma~\ref{lem1rc}:

\begin{lemma} \label{lem2rc}
Let $V \in L_{j+1}^1$ with $j\ge 1$. Then $\frac{\partial^l}{\partial k^l}\big(\frac{h_\pm (x,k)-h_\pm (x,0)}{k}\big)$ as well as $\frac{\partial^l}{\partial k^l}\big(\frac{h'_\pm (x,k)-h'_\pm (x,0)}{k}\big)$ are contained in $\mathcal A$ for $0 \le l \le j-1$. Moreover, for $\pm x \geq 0$  the $\mathcal A$-norms of these expressions do not depend on $x$.
\end{lemma}

\begin{proof}
Using \eqref{Jost} and Fubini, a little calculation shows
\begin{align*}
\frac{h_\pm (x,k)-h_\pm (x,0)}{k}=\pm 2 \I \int_{0}^{\pm \infty} & B_\pm(x,z) \left( \int_{0}^{z} \E^{\pm 2 \I ky} dy \right) dz\\
=\pm 2 \I \int_{0}^{\pm \infty} \E^{\pm 2 \I ky}  \left( \int_{y}^{\pm \infty} B_\pm(x,z) dz \right) dy&=2 \I \int_{0}^{\pm \infty} K_\pm(x,y) \E^{\pm 2 \I ky} dy.
\end{align*}
Now after differentiating with respect to $k$ the claim follows by \eqref{est1}. For the second item we can proceed exactly in the same way.
\end{proof}

Similarly,

\begin{lemma} \label{lem4rc}
Let $V \in L_{j+1}^1$ with $j\ge 1$. Then $\frac{d^l}{dk^l}\big(\frac{\Psi_\pm(k)-\Psi_\pm(0)}{k}\big)\in\mathcal A$ for $0\le l \le j-1$.
\end{lemma}

\begin{proof}
This follows as in the previous lemma using the estimate for $H_\pm$ from Lemma~\ref{lem3rc}.
\end{proof}

Combining the last two lemmas we obtain:

\begin{theorem} \label{mainrc}
Let $V \in L_{j+1}^1$ with $j\ge 1$. Then $\frac{d^l}{dk^l}\big(\frac{T(k)-T(0)}{k}\big)$ and $\frac{d^l}{dk^l}\big(\frac{R_\pm(k)-R_\pm(0)}{k}\big)$ are elements of $\mathcal A$ for $0\le l \le j-1$.
\end{theorem}

\section{Dispersive decay estimates}\label{sec:dc}

In this section we prove the integrable dispersive decay estimate \eqref{resdec} for the Schr\"odinger equation \eqref{Schr} in the resonant case. For the one-parameter group of \eqref{Schr} the spectral theorem  and Stone's formula imply
\begin{equation}\label{PP}
   \E^{-\I tH}P_{ac}
   =\frac 1{2\pi \I}\int\limits_{0}^{\infty}
   \E^{-\I t\omega}(\cR(\omega+\I 0)- \cR(\omega-\I 0))\,d\omega,
\end{equation}
where $\cR(\omega)=(H-\omega)^{-1}$ is the resolvent of the Schr\"odinger operator $H$,
and the limit is understood in the strong sense \cite[Problem 4.3]{tschroe}.
Given the Jost solutions, we can express the kernel of the resolvent $\cR(\omega)$ for $\omega=k^2\pm \I 0$, $k>0$, as \cite[Lemma 9.7]{tschroe}
\begin{equation}\label{RJ1-rep}
[\cR(k^2\pm \I 0)](x,y) = - \frac{f_+(y,\pm k) f_-(x,\pm k)}{W(\pm k)}
=\mp\frac{f_+(y,\pm k) f_-(x,\pm k)T(\pm k)}{2\I k}
\end{equation}
 for all  $x\leq y$ (and the positions of $x,y$ reversed if $x>y$).
Therefore, in  the case  $x\le y$, the integral kernel of $\E^{-\I tH}P_{ac}$ is given by
\begin{align}
[\E^{-\I tH}P_{ac}](x,y)= \frac{1}{2\pi} \int_{-\infty}^{\infty} \E^{-\I(t k^2-|y-x|k)}h_+(y,k) h_-(x,k) T(k)dk, \label{integr}
\end{align}
where the integral has to be understood as an improper  integral.
Another result that we need in order to obtain our decay estimates is the following variant of the van der Corput lemma \cite[Lemma~5.4]{EKMT}:

\begin{lemma}\label{vdcorputcontinuous}
Consider the oscillatory integral $I(t) = \int_a^b \E^{\I t \phi(k)} f(k) dk$,
where $\phi(k)$ is a real-valued function. If $\phi''(k)\not =0$ in [a,b] and $f\in\cA_1$, then we have
$|I(t)| \le C_2 [t\min_{a\le k\le b}|\phi''(k)|]^{-1/2}\|f\|_{\mathcal A_1}$ for $t\ge 1$,
where $C_2\le 2^{8/3}$ is the optimal constant from the van der Corput lemma.
\end{lemma}

Now we come to the proof of our main Theorem \ref{thm:mainr}. We first give an alternate representation of our projection operator $(4 \pi \I t)^{-\frac{1}{2}}P_0$:

\begin{lemma}[\cite{gold}] \label{lem5rc}
The integral kernel of $(4 \pi \I t)^{-\frac{1}{2}}P_0$, which is (per definition) given by $(4 \pi \I t)^{-\frac{1}{2}} f_0(x) f_0(y)$, can also be written in the form
\[
\frac{1}{2 \pi} \int_{-\infty}^{\infty} \E^{-\I tk^2} T(0) f_-(x,0) f_+(y,0) dk.
\]
\end{lemma}

\begin{proof}
It is clear that $f_0(x) = c_\pm f_\pm(x,0)$ and by our normalization $c_-^{-2}+c_+^{-2} = 2$. Using \eqref{def:gam} we have $c_-= \gamma c_+$ and
hence $c_\pm^2=\frac{1+\gamma^{\mp1}}{2}$. Moreover, \eqref{eq:trz} implies $c_+ c_- T(0)=1$, and using $\frac{1}{2 \pi} \int_{-\infty}^{\infty} \E^{-\I tk^2} dk=(4 \pi \I t)^{-\frac{1}{2}}$
the claim follows.
\end{proof}

Finally we have all the ingredients needed to obtain Theorem \ref{thm:mainr}:

\begin{proof}[Proof of Theorem~\ref{thm:mainr}]
For the kernel  of $\E^{-\I tH}P_{ac}$ we use \eqref{integr}.
 Then by Lemma~\ref{lem5rc}, the kernel $G(x,y,t)$ of $\E^{-\I tH}P_{ac}-(4 \pi \I t)^{-\frac{1}{2}}P_0$ can now be written as
\[
G(x,y,t)=\frac{1}{2\pi} \int_{-\infty}^{\infty} \E^{-\I t k^2} (\E^{\I |y-x|k}h_+(y,k) h_-(x,k) T(k)-h_+(y,0)h_-(x,0)T(0))dk.
\]
 Integrating this formula by parts, we obtain
\[
G(x,y,t)=\frac{1}{4 \pi \I t} \int_{-\infty}^{\infty} \E^{-\I t k^2} S(x,y,k)dk,
\]
where \[S(x,y,k)=\frac{\partial}{ \partial k} \left( \frac{\E^{\I |y-x|k}h_+(y,k) h_-(x,k) T(k)- h_+(y,0) h_-(x,0) T(0)}{k} \right).\]
Now we apply Lemma~\ref{vdcorputcontinuous} to get the desired $t^{-\frac{3}{2}}$ time-decay.
So in order to finish our proof, it remains to bound the $\mathcal A$-norm of $S(x,y,k)$ which follows from the lemma below.
\end{proof}

\begin{lemma} \label{lastrc}
Assume $V \in L_3^1$. Then \[\|S(x,y,\cdot)\|_{\mathcal A}\leq  C(|x|+|y|)^2.\]
\end{lemma}

\begin{proof}
We assume w.l.o.g.\ $x \le y$ and distinguish the cases (i) $x \le 0 \le y$, (ii) $0 \le x \le y$ and (iii) $x \le y \le 0$.
Introduce the function $g(x,y,k):=T(k) h_+(y,k) h_-(x,k)$. Then $S(x,y,k) $ can be written as:
\begin{align*}
S(x,y,k)=&\frac{\partial}{\partial k}\left( \frac{\E^{\I (y-x) k}-1}{k}\right) g(x,y,k)\\
&+ \frac{\E^{\I (y-x) k}-1}{k} \frac{\partial}{\partial k}g(x,y,k)+\frac{\partial}{\partial k} \left( \frac{g(x,y,k)-g(x,y,0)}{k} \right).
\end{align*}
The $\mathcal A$-norm of $\frac{\E^{\I (y-x) k}-1}{k}$ is bounded by $C (|x|+|y|)$ and that of its derivative by $C (|x|+|y|)^2$. So it remains to consider the $\mathcal A_1$-norms of $g(x,y,k)$ and $\frac{\partial}{\partial k} g(x,y,k)$, and $\mathcal A$-norm of $\frac{\partial}{\partial k} P(x,y,k)$, where
\[P(x,y,k)= \frac{g(x,y,k)-g(x,y,0)}{k}.\]
We start with case (i). Then $g(x,y,k)\in\mathcal A_1$ with $\mathcal A_1$-norm independent of $x$ and $y$. After applying the product rule, Lemmas \ref{lem1rc} and  \ref{mainrc} imply $\|\frac{\partial}{\partial k} g(x,y,k)\|_{\mathcal A}\leq C$. Moreover,
\begin{align*} P(x,y,k)&=\frac{T(k)-T(0)}{k}h_+(y,k)h_-(x,k) \\
&+\frac{h_+(y,k)-h_+(y,0)}{k}h_-(x,k)T(0)+ \frac{h_-(x,k)-h_-(x,0)}{k}h_+(y,0)T(0).
\end{align*}
Taking here the derivative with respect to $k$ and invoking Lemma \ref{lem1rc}, Lemma \ref{lem2rc}, and Lemma \ref{mainrc}, we are done in this case.
In the cases (ii) and (iii) we use the scattering relations \eqref{scat-rel} to see that the following representations are valid:
\[
g(x,y,k)=\begin{cases}
 h_+(y,k)\left(R_+(k)h_+(x,k)\E^{2\I xk} + h_+(x,-k)\right), & 0\leq x\leq y, \\[2mm]
 h_-(x,k)\left(R_-(k)h_-(y,k)\E^{-2\I yk} + h_-(y,-k)\right), & x\leq y\leq 0.
\end{cases}
\]
Thus $g(x,y,k)$ has an $\mathcal A_1$-norm independent of $x$ and $y$, since for any function $f(k)\in\mathcal A$ and any real $s$ we have $f(k)\E^{\I k s}\in \mathcal A$ with the norm independent of $s$.
If we take the derivative with respect to $k$, again everything is contained in $\mathcal A$ by Lemma \ref{lem1rc} and Lemma \ref{mainrc}, however we get additional terms from the derivatives of $\E^{2\I xk}$ and $\E^{2\I yk}$. So it follows that the $\mathcal A$-norm of $\frac{\partial}{\partial k} g(x,y,k)$ is at most proportional to $|x|$ or $|y|$ respectively.
Finally, let us have a look at $\frac{\partial}{\partial k}P(x,y,k)$. We observe that in case (ii) one can represent $P$ as 
\begin{align*}
P(x,y,k)=\frac{h_+(y,k)-h_+(y,0)}{k}h_+(x,-k) &+ \frac{h_+(x,-k)-h_+(x,0)}{k}h_+(y,0)\\
+\frac{R_+(k)-R_+(0)}{k} h_+(x,k) h_+(y,k) \E^{2\I x k}&+\frac{\E^{2\I x k}-1}{k} R_+(0)h_+(x,k)h_+(y,k)\\
+\frac{h_+(x,k)-h_+(x,0)}{k}h_+(y,k)R_+(0) &+ \frac{h_+(y,k)-h_+(y,0)}{k}h_+(x,0)R_+(0).
\end{align*}
Here again every summand is an element of $\mathcal A$ by Lemma \ref{lem1rc}, Lemma \ref{lem2rc}, and Theorem \ref{mainrc}. Since the derivative of $\frac{\E^{2\I x k}-1}{k}$ also occurs here, we conclude that the $\mathcal A$-norm of $\frac{\partial}{\partial k} P(x,y,k)$ is at most proportional to $|x|^2$. From the same reasons in the case (iii) this derivative will be proportional to $|y|^2$.
\end{proof}

\noindent{\bf Acknowledgments.}
I.E.~is indebted to the Department of Mathematics at the University of Vienna for its hospitality and support during the
fall of 2014, where some of this work was done. We thank Fritz Gesztesy for discussions on this topic.

\end{document}